\documentclass[reqno,10pt]{amsart}
\usepackage{amsfonts}

\usepackage{latexsym,amssymb,amsthm,amsmath,amscd}
\usepackage{latexsym,amssymb,amsthm,amsmath}
\usepackage{graphics,amscd}
\usepackage{amssymb}
\usepackage{eucal}
\usepackage{amsmath}

\theoremstyle{plain}
\newtheorem{theorem}{Theorem}[section]
\newtheorem*{Theorem B}{Theorem B}
\newtheorem*{Theorem A}{Theorem A}

\newtheorem{lemma}{Lemma}[section]

\newtheorem{corollary}{Corollary}[section]

\newtheorem{definition}{Definition}[section]

\numberwithin{equation}{section}

\theoremstyle{remark}
\newtheorem{remark}{Remark}[section]
 \numberwithin{equation}{section}

\def\<{\left < }
\def\>{\right >}
\def\({\left ( }
\def\){\right )}

\def\e{\eqref}
\def\p{\partial }
\def\x{{\bf x}}

\begin{document}
\title[Yamabe and quasi-Yamabe solitons]{Yamabe and quasi-Yamabe solitons on Euclidean submanifolds}
\author[B.-Y. Chen]{Bang-Yen Chen}
\address{Department of Mathematics, Michigan State University, 619 Red Cedar
Road, East Lansing, Michigan 48824--1027, U.S.A.}
\email{chenb@msu.edu}
\author[S. Deshmukh]{Sharief Deshmukh}
\address{Department of Mathematics, College of science, King Saud University
P.O. Box-2455 Riyadh-11451, Saudi Arabia}
\email{shariefd@ksu.edu.sa}

\begin{abstract} 
In this paper we initiate the study of Yamabe and quasi-Yamabe solitons on Euclidean submanifolds whose soliton  fields are  the tangential components of their position vector fields. Several fundamental results of such solitons were proved. In particular, we  classify such Yamabe and quasi-Yamabe solitons on Euclidean hypersurfaces. 

\end{abstract}


\maketitle

\section{Introduction}

 The Yamabe flow was
introduced by R. Hamilton at the same time as the Ricci flow (cf. \cite{H98}). It deforms a given manifold
by evolving its metric according to
\begin{align}\label{1.1} \frac{\p}{\p t}g(t)=-R(t) g(t),\end{align}
where $R(t)$ denotes the scalar curvature of the metric $g(t)$.
Yamabe solitons correspond to self-similar solutions of the Yamabe flow. 

 In dimension $n = 2$ the Yamabe
flow is equivalent to the Ricci flow (defined by $\frac{\p}{\p t}g(t)=-2 \rho(t)$, where $\rho$ stands for the
Ricci tensor). However in dimension $n>2$ the Yamabe and Ricci flows do not agree, since
the first one preserves the conformal class of the metric but the Ricci flow does not in general.

A Riemannian manifold $(M, g)$ is a {\it Yamabe soliton} if it admits a vector field $X$ such that
\begin{equation}\label{1.2}\frac{1}{2}{\mathcal L}_{X}g= (R-\lambda) g,\end{equation} 
where ${\mathcal L}_{X}$ denotes the Lie derivative in the direction of the vector field $X$ and $\lambda$ is a real number. Moreover, a vector field $X$ as in the definition is called a {\it soliton field} for
$(M, g)$. In the following, we denote the Yamabe soliton satisfying \e{1.2} by $(M,g,X,\lambda)$.
 A Yamabe soliton is said to be {\it shrinking, steady or expanding} if
it admits a soliton field for which, respectively, $\lambda>0,\, \lambda=0$ or $\lambda<0$. 
 
We call a Riemannian manifold $(M,g)$ a {\it quasi-Yamabe soliton} if  it admits a vector field $X$ such that
\begin{equation}\label{1.3}\frac{1}{2}{\mathcal L}_{X}g= (R-\lambda) g+ \mu X^\#\otimes X^\#,\end{equation} 
for some constant $\lambda$ and some function $\mu$, where $X^\#$ is the dual 1-form of $X$. The vector field $X$  is also called a {\it soliton field} for the quasi-Yamabe soliton. We denote the quasi-Yamabe soliton satisfying \e{1.3} by $(M,g,X,\lambda,\mu)$.

When $X=\nabla f$ is a gradient field. then \e{1.3} becomes
\begin{equation}\label{1.4}\nabla^2f= (R-\lambda) g+ \mu df\otimes df,\end{equation} 
which is nothing but a generalized quasi-Yamabe gradient soliton (see \cite{GH14,LP16}), where $\nabla^2f$ denotes the Hessian of $f$.

For a submanifold $M$ of a Euclidean $m$-space $\mathbb E^m$, the most natural tangent vector field of $M$ is the tangential component of the position vector field $\x$ of $M$ (cf. for instance \cite{C17a,C17b}). Ricci solitons on Euclidean submanifolds arisen from such a vector field have been studied recently by the authors in \cite{CD2,CD3}. 

In this paper we initiate the study of Yamabe and quasi-Yamabe solitons on Euclidean submanifolds whose soliton fields are the tangential components of their position vector fields.  Several fundamental results of such solitons were proved. 
In particular, we classify Yamabe and quasi-Yamabe solitons on Euclidean hypersurfaces whose potential fields are the tangential component of their position vector fields.

\section{Basic definitions and formulas}

For general references on Riemannian submanifolds, we refer to \cite{book73,book11,book15}.

 Let  $(M,g)$ be an $n$-dimensional Riemannian manifold. For an orthonormal basis $e_1,\ldots,e_n$ of a tangent space $T_pM$ at $p\in M$.
  Denote the sectional curvature of a plane section spanned by  $e_i$ and $e_j$ $(i\ne j)$ by $K_{ij}$. Then the scalar curvature $R$ of $M$ is given by
  \begin{align}\label{2.1} R=\sum_{1\leq i\ne j\leq n} K_{ij}. \end{align}
 
 Let $\phi:(M,g) \to \mathbb E^m$ an isometric immersion of a Riemannian $n$-manifold $(M,g)$ into a Euclidean $m$-space $(\mathbb E^m,\tilde g)$.      
Denote by $\nabla$ and $\tilde\nabla$ the Levi-Civita connections on $(M,g)$ and $(\mathbb E^m,\tilde g)$, respectively. 

For vector fields $X,Y$ tangent to $M$ and $\eta$ normal to $M$, the formula of Gauss and the formula of Weingarten are given respectively by \begin{align} &\label{2.2}\tilde \nabla_XY=\nabla_XY+h(X,Y), \;\;
\\& \label{2.3}\tilde \nabla_X \eta=-A_\eta X+D_X\eta,\end{align} 
where $\nabla_X Y$ and $h(X,Y)$ are the tangential and the normal components of $\tilde\nabla_X Y$. Similarly,  $-A_\eta X$  and  $D_X\eta$ are the tangential and normal components of  $\tilde \nabla_X \eta$. These two formulas define the second
fundamental form $h$, the shape operator $A$, and the normal connection $D$ of $M$ in the ambient space $\mathbb E^m$. 
 
  It is well-known that each  $A_{\eta}$ is a self-adjoint endomorphism. The shape operator $A$ and the second fundamental form $h$ are related by
 \begin{align} &\label{2.4}\tilde g(h(X,Y),\eta)=g(A_{\eta}X,Y).\end{align}
  
  The {\it mean curvature vector}  $H$ of $M$ in $\mathbb E^m$ is defined by \begin{align}\label{2.5} H=\(\frac{1}{n}\){\rm trace}\, h.\end{align}
A submanifold $M$ is called {\it minimal} if its mean curvature vector field vanishes identically. It is called {\it totally umbilical} if the second fundamental form satisfies 
$ h(X,Y)=g(X,Y)H$ for tangent vectors $X,Y$. 
A hypersurface of a Euclidean $(n+1)$-space $\mathbb E^{n+1}$ is called a {\it quasi-umbilical hypersurface} if its shape operator has an eigenvalue $\kappa$ of multiplicity $mult(\kappa)\geq n-1$ (cf. \cite[page 147]{book73}). On the subset $U$ of $M$ on which $mult(\kappa)=n-1$, an eigenvector with eigenvalue of multiplicity one is called a {\it distinguished direction} of the quasi-umbilical hypersurface.

The equations of Gauss is given by
\begin{align} &\label{2.6} g(R(X,Y)Z,W) =  \tilde g(h(X,W),h(Y,Z)) - \tilde g(h(X,Z),h(Y,W))\end{align}
for vectors $X,Y,Z,W$  tangent to $M$

For a function $f$ on $M$, we denote by $\nabla f$ and $H^f$ the gradient of $f$ and the Hessian of $f$, respectively. Thus we have
\begin{align}\label{2.7} & g(\nabla f,X)=Xf,
\\&\label{2.8} H^f(X,Y)=XYf-(\nabla_X Y)f.\end{align}

Let $(M,g)$ be a Riemannian $m$-manifold. Associated with the Ricci tensor $Ric$, define a $(1,1)$-tensor $Q$ by $$g(Q(X),Y)=Ric(X,Y).$$
The {\it Weyl conformal curvature tensor} $C$ is a tensor field of type $(1,3)$ defined by
\begin{equation}\begin{aligned}\notag & C(X,Y)Z=R(X,Y)Z+\text{\small$\frac{1}{m}$}\{Ric(X,Z)Y-Ric(Y,Z)X\\&\hskip.1in +\<X,Z\>QY -\<Y,Z\>QX\}-\text{\small$\frac{2\tau}{m(m+1)}$}\{\<X,Z\>Y-\<Y,Z\>X\}.\end{aligned}\end{equation}

A well-known result of H. Weyl \cite{W18} states that a Riemannian manifold $M$ of dimension $\geq 4$ is conformally flat if and only if the conformal curvature tensor $C$ vanishes identically.

\section{Euclidean Submanifolds as Yamabe solitons}

For an isometric immersion $\phi:(M,g)\to \mathbb E^{m}$ of a Riemannian $n$-manifold $(M,g)$ into a Euclidean $m$-space $\mathbb E^{m}$, we denote by $\x^T$ and $\x^N$ the tangential and normal components of the position vector field $\x$ of $M$ in $\mathbb E^{m}$, respectively. 
So, we have 
\begin{align} \label{3.1} \x=\x^T+\x^N.\end{align}

\begin{theorem}\label{T:3.1} A  Euclidean submanifold $(M,g)$ of $\mathbb E^{m}$ is a Yamabe soliton with $\x^T$ as its soliton field if and only if the second fundamental form $h$ of $M$ satisfies
\begin{align} \label{3.2}\tilde g(h(V,W),\x^N)=(R-\lambda-1)g(V,W)\end{align}
for vectors $V,W$ tangent to $M$, where $R$ is the scalar curvature of $M$ and $\lambda$ is a constant.
\end{theorem} 
\begin{proof} Let $\phi:(M,g)\to \mathbb E^{m}$ denote the isometric immersion. It is well-known that the position vector field $\x$ of $M$ in $\mathbb E^{m}$ is a concurrent vector field, i.e., $\x$ satisfies
\begin{align} \label{3.3} \tilde\nabla_V\x=V,\end{align}
for any vector $V$ tangent to $M$.

It follows from \e{3.1}, \e{3.3} and formulas of Gauss and Weingarten that
\begin{equation}\begin{aligned} \label{3.4} V&=\tilde \nabla_V \x^T+\tilde \nabla_V \x^N=
 \nabla_V \x^T+h(V,\x^T)-A_{\x^N}V+D_V \x^N
\end{aligned}\end{equation} for any $V$ tangent to $M$.
By comparing the tangential and normal components from \e{3.4} we find
\begin{align} \label{3.5}& \nabla_V \x^T=A_{\x^N}V+V,
\\&\label{3.6} h(V, \x^T)=-D_V \x^N.\end{align}
From the definition of Lie derivative and \e{3.5} we obtain
\begin{equation}\begin{aligned} \label{3.7}({\mathcal L}_{\x^T}g)(V,W)&=g(\nabla_V \x^T,W)+g(\nabla_W \x^T,V) \\&=2g(V,W)+2 g(A_{\x^N}V,W)
\\&=2g(V,W)+2\tilde g(h(V,W),\x^N)
\end{aligned}\end{equation} for $V,W$ tangent to $M$. Consequently,  by applying \e{1.2} and \e{3.6}, we conclude that $(M,g)$ is a Ricci soliton with $\x^T$ as its soliton field if and only if \e{3.2} holds identically for some constant $\lambda$.
\end{proof}

An important application of Theorem \ref{T:3.1} is the following.

\begin{theorem}\label{T:3.2} Let $(M,g)$ be a Riemannian manifold. Then an isometric immersion $\phi: (M,g)\to S^{m-1}_o(r)\subset \mathbb E^m$ of $M$ into the hypersphere $S^{m-1}_o(r)$ of radius $r$ with center at the origin $o$ is a Yamabe soliton with $\x^T$ as its soliton field if and only if $(M,g)$ has constant scalar curvature $R$.
\end{theorem}
\begin{proof}  Let $(M,g)$ be a Riemannian manifold. If $\phi: (M,g)\to S^{m-1}_o(r)\subset \mathbb E^m$ is an isometric immersion  of $M$ into $S^{m-1}_o(r)$, then we have $\x=\x^N$.
Also in this case it follows from \cite[Lemma 3.5, page 60]{book11} that the second fundamental form of $M$ in $\mathbb E^m$ satisfies \begin{align}\label{3.8} h(V,W)=h'(V,W)-\frac{g(V,W)}{r^2}\x\end{align}
for vectors $V,W$ tangent to $M$, where $h'$ denotes the second fundamental form of $M$ in $S^{m-1}_o(r)$.
Clearly, \e{3.8} implies
\begin{align}\label{3.9} \tilde g(h(V,W),\x^N)=-g(V,W).\end{align}
Thus condition \e{3.2} in Theorem \ref{T:3.1} holds if and only if  $R=\lambda$ holds. Therefore $(M,g,\x^T,\lambda)$ with $\lambda=R$ is a Yamabe soliton if and only if $(M,g)$ has constant scalar curvature $R$.
\end{proof}

\begin{remark} Theorem \ref{T:3.2} implies that there exist ample examples of Yamabe solitons with $\x^T$ as the soliton fields.
\end{remark}

The next result classifies Yamabe solitons on Euclidean hypersurfaces with $\x^T$ as the soliton fields.

\begin{corollary}\label{C:3.1} Let $(M,g)$ be a Euclidean hypersurface of $\mathbb E^{n+1}$. Then $(M,g,\x^T,\lambda)$ is a Yamabe soliton if and only if either
\begin{itemize}
\item[{\rm (1)}] $\lambda=-1$ and $M$ is an open part of a hyperplane of $\mathbb E^{n+1}$, or
    
\item[{\rm (2)}] $\lambda=R>0$ and $M$ is an open part of a hypersphere of $\mathbb E^{n+1}$ centered the origin.\end{itemize}
\end{corollary} 
\begin{proof} Let $M$ be a Euclidean hypersurface of $\mathbb E^{n+1}$. Assume that $(M,g,\x^T,\lambda)$ is a Yamabe soliton with $\x^T$ as its soliton field. Then we have \e{3.2} by Theorem \ref{T:3.1}.

\vskip.1in
{\it Case} (a): $\x^N\equiv 0$. In this case, the position vector field $\x$ is always tangent to the hypersurface $M$. Thus $M$ is an open portion of a hyperplane containing the origin of $\mathbb E^{n+1}$. Hence, by equation \e{2.6} of Gauss, $M$ is a flat space immersed as a totally geodesic hypersurface in $\mathbb E^{n+1}$. Therefore it follows from \e{3.2} that $\lambda=-1$. Consequently, the Yamabe soliton $(M,g)$ is an expanding one. This gives Case (1) of the corollary.

\vskip.1in
{\it Case} (b): $\x^N\equiv \x$. In this case, the position vector field $\x$ is normal to the hypersurface $M$ everywhere. Thus $M$ is an open portion of a hypersphere of radius, say $r$, centered at the origin. So, in this case we find from (see, e.g. \cite[Lemma 3.5, page 60]{book11}) that
\begin{align} \label{3.10} h(V,W)=-\frac{g(V,W)}{r^2}\x\end{align}
for $V,W$ tangent to $M$. After substituting \e{3.10} into \e{3.2} we obtain $\lambda=R>0$. Consequently, the Yamabe soliton is shrinking. This gives Case (2)  of the corollary.

\vskip.1in
{\it Case} (c): $\x^N\ne 0,\x$. It follows from \e{3.2} that $M$ is totally umbilical in $\mathbb E^m$. Hence the scalar curvature $R$ is constant. Thus  \e{3.2} gives
\begin{align} \label{3.11}\tilde g(h(U,U),\x^N)=R-\lambda-1=constant\end{align}
for any unit vector $U$ tangent to $M$.

Now, suppose that $M$ is totally geodesic in $\mathbb E^{n+1}$. Then \e{3.11} reduce to $\lambda=-1$.
Hence we obtain Case (1) again.

First, let us assume that $M$ is totally umbilical, but not totally geodesic  in $\mathbb E^{n+1}$. Then  $M$  is contained in a hypersphere with radius, say $r$, centered at $\x_o\ne 0$. Thus we have (see, e.g. \cite[Lemma 3.5, page 60]{book11}):
\begin{align} \label{3.12} h(U,U)=-\frac{\x-\x_o}{r^2}\end{align}
for any unit vector $U$ tangent to $M$. Now, after substituting \e{3.12} into \e{3.2} we obtain
\begin{align} \notag \tilde g(\x_o-\x,\x^N)=constant\end{align}
It is easy to see that $\x-\x_o=rN$ and $\x^N=|\x^N| N$, where $N$ is a unit normal vector field of $M$. Therefore $|\x^N|$ is constant on $M$ which is impossible, since the center of the hypersurface is not  the origin of $\mathbb E^{n+1}$.
\end{proof}

A unit normal vector field $\xi$ of a Euclidean submanifold $M$ is called a {\it parallel} (resp., {\it nonparallel}\/) normal section if $D\xi=0$ (resp., $D\xi\ne 0$) everywhere on $M$ (cf. \cite{book73,CY71,CY73}). 

\begin{corollary}\label{C:3.2} Let $(M,g)$ be an $n$-dimensional submanifold of $\mathbb E^{n+2}$ with $n>3$ and $\x^N\ne 0$. Assume that $(M,g,\x^T,\lambda)$ is a Yamabe soliton. Then we have:
\begin{itemize}
\item[{\rm (1)}] If $\frac{\x^N}{|\x^N|}$ is a nonparallel normal section, then $(M,g)$ is a conformally flat space.  Moreover, in this case $M$ is the locus of $(n-1)$-spheres.  

\item[{\rm (2)}] If $\frac{\x^N}{|\x^N|}$ is a parallel normal section, then $(M,g)$ lies either a hyperplane or in a hypersphere of $\mathbb E^{n+2}$.
\end{itemize}
\end{corollary}
\begin{proof}  Let $(M,g)$ be an $n$-dimensional submanifold of $\mathbb E^{n+2}$ with $n>3$ and $\x^N\ne 0$. If $(M,g,\x^T,\lambda)$ is a Yamabe soliton, then it follows from Theorem \ref{T:3.1} that $M$ is umbilical with respect the normal direction $\x^N$.

\vskip.05in
{\it Case} (i): $\frac{\x^N}{|\x^N|}$ {\it is a nonparallel normal section}. It follows from  \cite[Theorem 3]{CY73} that $(M,g)$ is a conformally flat space. Moreover, from  \cite[Theorem 4]{CY73} we know that the hypersurface  is a locus of $(n-1)$-spheres in $\mathbb E^{n+1}$.

\vskip.05in
{\it Case} (ii):  $\frac{\x^N}{|\x^N|}$ {\it is a parallel normal section}. It follows from  \cite[Theorem 3.3]{CY71} that $M$ lies either in a hyperplane or in a hypersphere of $\mathbb E^{n+2}$.
\end{proof}

\section{Euclidean submanifolds as quasi-Yamabe solitons}

For quasi-Yamabe solitons we have the following.

\begin{theorem}\label{T:4.1} A  Euclidean submanifold $(M,g)$ of $\mathbb E^{m}$ is a quasi-Yamabe soliton with $\x^T$ as its soliton field if and only if the second fundamental form $h$ of $(M,g)$ satisfies
\begin{align} \label{4.1}\tilde g(h(V,W),\x^N)=(R-\lambda-1)g(V,W)+\mu g(\x^T,V)g(\x^T,W)\end{align}
for vectors $V,W$ tangent to $M$, where  $\lambda$ is a constant, $\mu$ is a function and $R$ is the scalar curvature of $M$.
\end{theorem}
\begin{proof} By applying \e{1.3} and \e{3.7}, this theorem can be proved in the same way as Theorem \ref{T:3.1}.
\end{proof}
 
  In \cite{Y44}, K. Yano extended concurrent vector fields to torse-forming vector fields. According to K. Yano, a vector field $v$ on a  Riemannian manifold $M$ is called a {\it torse-forming vector field} if it satisfies
\begin{align}\label{4.2} \nabla_X v=\varphi X+\alpha(X) v,\;\; \forall X\in TM,\end{align}
 for a function $\varphi$ and a 1-form $\alpha$ on $M$. The 1-form $\alpha$ is called the {\it generating form} and the function $\varphi$ is called the {\it conformal scalar} (see \cite{Mihai}). 
  A torse-forming vector field $v$ is called {\it proper torse-forming} if the 1-form $\alpha$ is nowhere zero on a dense open subset of $M$.
  
  By a cone in $\mathbb E^{m}$ with vertex at the origin we mean a ruled submanifold generated by a family of lines passing through the origin. A submanifold of $\mathbb E^{m}$ is called a {\it conic submanifold} with vertex at the origin if it is an open portion of a cone with vertex at the origin.
  
  The following two lemmas can be found in \cite{c16} (see also \cite{c17}).

\begin{lemma} \label{L:4.1} Let $\,x:(M,g)\to {\mathbb E}^{m}$ be an isometric immersion of a Riemannian
manifold into a Euclidean $m$-space ${\mathbb E}^{m}$. Then $\x=\x^T$ holds identically if and only if  $M$ is a conic submanifold with the vertex at the origin. \end{lemma}

\begin{lemma} \label{L:4.2} Let $\,x:(M,g)\to {\mathbb E}^{m}$ be an isometric immersion of a Riemannian manifold into ${\mathbb E}^{m}$. Then $\x=\x^N$ holds identically if and only if  $M$ lies in a hypersphere centered at the origin. \end{lemma}

  In view of Lemma \ref{L:4.1} and Lemma \ref{L:4.2}, we make the following definition of {\it proper submanifolds} as in \cite{CW17}.

\begin{definition}\label{D:4.1} {\rm A Euclidean submanifold $M$ is called {\it proper} if both $\,\x^T$ and $\x^N$ are nowhere zero on some dense open subset of $M$.}
\end{definition}

\begin{theorem} \label{T:4.2} Let $M$ be a proper hypersurface  of $\mathbb E^{n+1}$.
If $(M,g,\x^T,\lambda,\mu)$ is a quasi-Yamabe soliton with $\mu\ne 0$, then $M$ is a quasi-umbilical hypersurface with $\x^T$ as its distinguished direction. Moreover, $\x^T$ is a torse-forming vector field.
\end{theorem}
\begin{proof} Under the hypothesis of the theorem, we obtain \e{4.1} from Theorem \ref{T:4.1}.
By combining \e{2.4} and \e{4.1} we find
\begin{align} \label{4.3}\ g(A_{\x^N}V,W)=(R-\lambda-1)g(V,W)+\mu g(\x^T,V)g(\x^T,W).\end{align}
It follows from \e{4.3} that
\begin{align} \label{4.4}& A_{\x^N}\x^T=(R-\lambda-1)\x^T+\mu |\x^T|^2 \x^T\\&
 \label{4.5}A_{\x^N}Z=(R-\lambda-1)Z\end{align}
for any  vector $Z\in TM$ with $g(Z,\x^T)=0$. Combining \e{4.4} and \e{4.5} yields
 \begin{align} \label{4.6}A_{\x^N}V=(\varphi-1)V+\alpha(V) \x^T,\;\; \forall V\in TM,\end{align}
with $\varphi=R-\lambda$ and $\alpha=\mu (\x^T)^\#$, where $(\x^T)^\#$ is the 1-form dual to $\x^T$.

It follows from \e{4.6} that $M$ is a quasi-umbilical hypersurface with $\x^T$ as its distinguished direction. Moreover, it follows from Lemma 4 of \cite{CV17} that $\x^T$ is a torse-forming vector field on $M$.
\end{proof}

A {\it rotational hypersurface} $M=S^{n-1}\times \gamma$ in $\mathbb E^{n+1}$ is an $O(n-1)$-invariant hyper-surface, where $S^{n-1}$ is a Euclidean sphere and 
\begin{align}\label{4.7}\gamma(u)=(g(u),u),\;\;   g(u)>0,\;\; u\in I,\end{align} is a plane curve defined on an open interval $I$,  the {\it profile curve}, and the $u$-axis is called the {\it axis of rotation}.  The rotational hypersurface $M$ can expressed as
\begin{align} \label{4.8} \x=(g(u) y_1 ,\cdots,g(u) y_n, u),\;\;  y_1^2+\cdots+y_n^2=1.\end{align}

The rotational hypersurfaces is called a {\it spherical cylinder}  if its profile curve $\gamma$ is a horizontal line segment  (i.e., $g=constant\ne 0$). And it is called a {\it spherical cone} if $\gamma$ is a non-horizontal line segment  (i.e., $g=cu$, $0\ne c\in {\bf R}$).
We only consider rotational hypersurfaces which contain no open portions of hyperspheres, spherical cylinders, or  spherical cones.

The next result follows immediately from Theorem \ref{T:4.2} and Theorem 6 of \cite{CV17}.
\begin{corollary} \label{C:4.1}  Let $M$ be a proper hypersurface  of $\mathbb E^{n+1}$. If $(M,g,\x^T,\lambda,\mu)$ is a quasi-Yamabe soliton with $\mu\ne 0$, then  $M$ is an open portion of a rotational hypersurface whose axis of rotation contains the origin.
\end{corollary}

\end{document}